\documentclass[reqno]{amsart}
\title{Uncertainty Principles on Nilpotent Lie groups}

\usepackage{amsmath,amssymb,amsthm,amsrefs,mathrsfs}
\numberwithin{equation}{section}
\theoremstyle{definition}

\newtheorem{thm}{\bf Theorem}[section]

\newtheorem{lem}[thm]{\bf Lemma}
\newtheorem{rem}[thm]{\sc Remark}
\newtheorem{prop}[thm]{\bf Proposition}

\newcommand{\R}{\mathbb{R}}

\newcommand{\h}{\mathbb{H}}

\newcommand{\C}{\mathbb{C}}

\newcommand{\g}{\mathfrak{g}}
\newcommand{\U}{\mathscr{U}}
\newcommand{\ch}{\mathcal{H}}

\newcommand{\W}{\mathcal{W}}
\newcommand{\B}{\mathcal{B}}

\DeclareMathOperator*{\tr}{tr}

\allowdisplaybreaks
\begin{document}
\begin{abstract}
Hardy's type uncertainty principle on connected nilpotent Lie groups for the Fourier transform is proved. An analogue of Hardy's theorem for Gabor transform has been established for connected and  simply connected nilpotent Lie groups . Finally Beurling's theorem  for Gabor transform is  discussed for  groups of the form $\mathbb{R}^n \times K$, where $K$ is a compact group.
\end{abstract}
\author[Jyoti Sharma]{JYOTI SHARMA}
\address{Department of Mathematics, University of Delhi, Delhi, 110007, India.}
\email{jsharma3698@gmail.com}

\author[Ajay Kumar]{AJAY KUMAR$^\ast$}
\address{Department of Mathematics, University of Delhi, Delhi, 110007, India.}
\email[Corresponding author]{akumar@maths.du.ac.in}
\thanks{$^\ast$Corresponding author, E-mail address: akumar@maths.du.ac.in}
\keywords{Hardy's type theorem, Fourier transform, Beurling theorem, Continuous Gabor transform, Nilpotent Lie group.}
\subjclass[2010]{Primary 43A32; Secondary 22D99; 22E25}
\maketitle

\section{Introduction}
One of the uncertainty  principles states that a non-zero integrable function $f$ on $\R$ and its Fourier transform $\widehat{f}$ cannot both simultaneously decay rapidly. For $f \in L^1(\R)$, the Fourier transform $\widehat{f}$ on $ \mathbb{R} $ is given by
\begin{flalign*}
\widehat{f}(\xi)=\int_{\R}{f(x)\ e^{-2\pi i \xi x}}\ dx.
\end{flalign*}
The following theorem of Hardy (see \cite{Har:33}) makes the above statement more precise: 
\begin{thm}(Hardy)
Let $f$ be a measurable function on $\R$ such that 
\renewcommand{\labelenumi}{(\roman{enumi})}
\begin{enumerate}
\item $|f(x)| \leq Ce^{-a\pi x^2}$, for all $x \in \R$, 
\item $|\widehat{f}(\xi)| \leq Ce^{-b\pi \xi^2}$, for all $\xi \in \R$,
\end{enumerate}
where $a$, $b$ and $C$ are positive constants. If $ab>1$, then $f=0$ a.e.
\end{thm}

\noindent Several analogues of the above result have been proved in the setting of $\R^n$, Heisenberg group $\h_n$, Heisenberg motion group $\mathbb{H}_n \ltimes K$, locally compact abelian groups, several  classes of solvable Lie groups, Euclidean motion group and nilpotent Lie groups (see \cite{ Bak:16,Bak:Kan:08, Bak:Kan:10,Kan:Kum:01,Par:Tha:09,Sar:Tha:05, Sit:Sun:Tha:95}). 
A generalization of the above result is as follows:
\begin{thm}{(Beurling)} Let $f$ be a square integrable function satisfying
\begin{align*}
\int_{\mathbb{R}}\int_{\mathbb{R}}|f(x)|\ |\widehat{f}(\xi)| e^{2\pi|x\cdot \xi|}\ dx\ d\xi < \infty.
\end{align*}
Then $f = 0$ a.e.
\end{thm}
\noindent The Beurling's theorem for Fourier transform  has been proved for several classes of nilpotent Lie groups (see \cite{Bak:Sal:08,Par:Sar,Sma1:11,Par:Tha:09}).
 For a detailed survey of the uncertainty principles for  Fourier transform, refer to \cite{Fol:Sit:97}.\\
The transformation of a signal using Fourier transform loses the information about time. Thus, in order to tackle such problems, \textit{a joint time-frequency analysis} was utilized. \textit{Gabor transform} is turned out to be one such tool. The approach used in this technique is cutting the signal into segments using a smooth window-function and then computing the Fourier transform separately on each smaller segment. It results in a two-dimensional representation of the signal. \\  
Let $\psi \in L^2(\R)$ be a fixed function usually called a \textit{window function}. The Gabor transform of a function $f\in L^2(\R)$ with respect to the window function $\psi$ is defined by $G_\psi f : \R \times \widehat{\R} \rightarrow \C$ as
\begin{flalign*}
G_\psi f(t,\xi)=\int_{\R}{f(x)\ \overline{\psi(x-t)}\ e^{-2\pi i \xi x}}\ dx,
\end{flalign*}
for all $(t,\xi)\in \R \times \widehat{\R}$. \\
In this paper, analogues of above uncertainty principles on nilpotent Lie groups for Fourier and Gabor transform have been studied.
Results obtained have been organized as follows: In section $3$, Hardy's type results for Fourier transform have been established for connected nilpotent Lie groups. The next  section deals, with an  analogue of  Hardy's theorem for Gabor transform. In section $5$, we prove Beurling's theorem for Gabor transform  for the groups of the form $\R^n \times K$, where $K$ is a compact group.
 \section{Preliminaries} 
\noindent For  a second countable, unimodular group $G$ of type I,  $dx$ will denote the Haar measure on $G$. Let $\widehat{G}$ be the dual space of G consisting of all irreducible unitary representation of $G$ equipped with Plancherel measure $d\pi$. For $f\in L^1\cap L^2(G)$, the Fourier transform $\widehat{f}$ of $f$ is an operator valued function on $\widehat{G}$ defined as 
\begin{align*}
\widehat{f}(\pi) = \int_{G} f(x)\ \pi(x)^* dx.
\end{align*}
Moreover, by Plancherel theorem \cite{Gro:03}, $\widehat{f}(\pi)$ is a Hilbert-Schmidt operator and satisfies the following 
\begin{align}
\int_{G}|f(x)|^2 dx = \int_{\widehat{G}}\|\widehat{f}(\pi)\|_{\text{HS}}^2\ d\pi.\label{planc}
\end{align}
  For each $(x,\pi)\in G\times \widehat{G}$, we define $\ch_{(x,\pi)}=\pi(x)\text{HS}(\mathcal{H}_\pi),
$ where $\pi(x)\text{HS}(\mathcal{H}_\pi)=\{\pi(x)T:T\in \text{HS}(\ch_\pi)\}$. One can see  that $\ch_{(x,\pi)}$ forms a Hilbert space with the inner product given by
\begin{flalign*}
\langle \pi(x)T,\pi(x)S\rangle_{\ch_{(x,\pi)}}=\tr{(S^\ast T)}=\langle T,S\rangle_{\text{HS}(\ch_\pi)}.
\end{flalign*}
Also, $\ch_{(x,\pi)}=\text{HS}(\ch_\pi)$ for all $(x,\pi)\in G\times \widehat{G}$. Let $\ch^2(G \times \widehat{G})$ denote the direct integral of $\{\ch_{(x,\pi)}\}_{(x,\pi)\in G\times \widehat{G}}$ with respect to the product measure $dx\ d\pi$. $\ch^2(G \times \widehat{G})$ forms a Hilbert space with the inner product given by
\begin{flalign*}
\langle F,K\rangle_{\ch^2(G\times \widehat{G})}=\int_{G\times \widehat{G}}{\tr{[F(x,\pi)K(x,\pi)^\ast]}}\ dx\ d\pi.
\end{flalign*}
Let $f\in C_c(G)$, the space of all continuous complex-valued functions on $G$ with compact support, and let $\psi$ be a fixed function in $L^2(G)$. For $(x,\pi) \in G\times \widehat{G}$, the continuous \textit{Gabor Transform} of $f$ with respect to the window function $\psi$ can be defined as a measurable field of operators on $G\times \widehat{G}$ by
\begin{flalign}
&&G_\psi f(x,\pi)&:=\int_{G}{f(y)\ \overline{\psi(x^{-1}y)}\ \pi(y)^\ast}\ dy. \label{opvalued}&
\end{flalign}
One can verify that $G_\psi f(x,\pi)$ is a Hilbert-Schmidt operator for all $x \in G$ and for almost all $\pi\in \widehat{G}$. We can extend $G_\psi$ uniquely to a bounded linear operator from $L^2(G)$ into a closed subspace of $\ch^2(G\times \widehat{G})$ which will be denoted by $G_\psi$. As in \cite{Far:Kam:12}, for $f_1, f_2 \in L^2(G)$ and  window functions $\psi_1$ and $\psi_2$, we have
\begin{flalign}
\langle G_{\psi_1}f_1, G_{\psi_2}f_2\rangle = \langle\psi_2,\psi_1 \rangle \langle f_1, f_2 \rangle. \label{GT-norm}
\end{flalign}
  
\section{Nilpotent lie group}
For a connected nilpotent Lie group $G$ with its simply connected covering group $\widetilde{G}$, let $\Gamma$ be a discrete subgroup of $\widetilde{G}$ such that $G=\widetilde{G}/\Gamma$. Denoting $\g$ by the Lie algebra of $G$ and $\widetilde{G}$, 
let $\B=\{X_1,X_2,\ldots,X_n\}$ be a strong Malcev basis of $\g$ through the ascending central series of $\g$. The norm function on $\g$ is defined as the Euclidean norm of $X$ with respect to the basis $\B$. Indeed, for $X=\sum_{j=1}^{n}{x_jX_j} \in \g$ with $x_j\in \R$,
\begin{flalign*}
&&\|X\|&=\Big(\sum_{j=1}^{n}{x_j^2}\Big)^{1/2}.&
\end{flalign*}
Define a `norm function' on $G$ by setting 
\begin{flalign*}
&&\|x\|&=\inf{\{\|X\| : X \in \g\ \text{such that}\ \exp_G{X}=x\}}. &
\end{flalign*}
The composed map,\ 
  $\R^n\rightarrow \g \rightarrow \widetilde{G}$
given by
\begin{flalign*}
&&&(x_1,\ldots,x_n)\rightarrow \sum_{j=1}^{n}{x_jX_j} \rightarrow \exp_{\widetilde{G}}\Big(\sum_{j=1}^{n}{x_jX_j}\Big) &
\end{flalign*}
is a diffeomorphism and maps the Lebesgue measure on $\R^n$ to the Haar measure on $\widetilde{G}$. In this manner, we identify the Lie algebra $\g$,  as a set with $\R^n$. Also, measurable (integrable) functions on $\widetilde{G}$ can be viewed as such functions on $\R^n$.\\
Let $\g^\ast$ be the vector space dual of $\g$ and $\{X_1^\ast,\ldots,X_n^\ast\}$ the basis of $\g^\ast$ which is dual to $\{X_1,\ldots,X_n\}$. Then, $\{X_1^\ast,\ldots,X_n^\ast\}$ is a Jordan-H\"older basis for the coadjoint action of $G$ on $\g^\ast$. We shall identify $\g^\ast$ with $\R^n$ via the map 
\begin{flalign*}
&&&\xi=(\xi_1,\ldots,\xi_n)\rightarrow \sum_{j=1}^{n}{\xi_jX_j^\ast}&
\end{flalign*}
and on $\g^\ast$ we introduce the Euclidean norm relative to the basis $\{X_1^\ast,\ldots,X_n^\ast\}$, i.e.
\begin{flalign*}
&&&\Big\|\sum_{j=1}^{n}{\xi_jX_j^\ast}\Big\| =\Big(\sum_{j=1}^{n}{\xi_j^2}\Big)^{1/2}=\|\xi\|.&
\end{flalign*}
Let $\U$ denote the Zariski open subset of $\g^\ast$ of generic elements under the coadjoint action of $\widetilde{G}$ with respect to the basis $\{X_1^\ast,\ldots,X_n^\ast\}$. Suppose that  $S$ is the set of jump indices, $T =\{1,\ldots,n\}\setminus S$ and $V_T=\R$-span$\{X_i^\ast: i \in T\}$. Then, $\W=\U \cap V_T$ is a cross-section for the generic orbits and $\W$ supports the Plancherel measure on $\widetilde{G}$. Every element of a connected nilpotent Lie group $G$ with non-compact centre  can be uniquely written as $(t,z,y),$ $t\in \mathbb{R}, z\in \mathbb{T}^d $ and $y\in Y, $ where $Y= \exp(\sum_{j=d+2}^{n}\mathbb{R}X_j)$. We now prove a generalization of the result proved in  \cite{Bak:Kan:08}.

\begin{thm}\label{simp}
Let $G$ be a connected nilpotent Lie group with non-compact center and  $f:G \to \mathbb{C}$ be a measurable function satisfying 
\begin{enumerate}
\item[(i)] $|f(t,z,y)| \leq C(1+|t|^2)^{N}e^{-\pi \alpha t^2}\phi(y)$ for  all $(t,z,y)\in G$   and some $\phi \in L^1\cap L^2(Y)$.
\item[(ii)] $\|\pi_{\xi}(f)\|_{\text{HS}} \leq C(1+\|\xi\|^2)^N e^{-\pi \beta \|\xi\|^2}$ for all $\xi \in \W$,
\end{enumerate}
  where $\alpha,\beta$ and $C$ are positive real numbers and $N$ is a non-negative integer. If $\alpha\beta > 1$, then $f=0$ a.e.
\end{thm}
 Let $K$ be a compact central subgroup of $G$ and $\chi$ be a character of $K$. For $f \in L^1(G)$, define $f_{\chi}:G \to \mathbb{C} $ by 
$$ f_{\chi}(t,z,y)=\int_{K}f(t,zk,y)\ \overline{\chi(k)}\ dk.$$
\begin{lem}\label{char} Let $G$ be a connected nilpotent Lie group with a compact central subgroup $K$ and $f$ be a measurable function on $G$ satisfying conditions (i) and (ii) of Theorem $\ref{simp} $. Then the function $f_{\chi}$ also satisfies these conditions.
\end{lem}
\begin{proof}
On normalizing the Haar measure on central subgroup $K$, we obtain 
\begin{align*}
|f_{\chi}(t,z,y)| &\leq \int_{K} C(1+t^2)^N\ e^{-\alpha \pi t^2} \varphi(y)\ dk\\
&=C(1+t^2)^N\ e^{-\alpha \pi t^2} \varphi(y).
\end{align*}
Also, $\pi_{\xi}(f_{\chi}) = \pi_{\xi}(f) \int_{K} \chi(k)\ \pi_{\xi}(k)\ dk$. If $\pi_{\xi}|_{K}$ is a multiple of some character of $K$ which is different from $\chi$, then by orthogonality relation of compact groups, we have
\begin{align*}
\int_{K} \chi(k)\ \pi(k)\ dk = 0.
\end{align*}
Thus, $\|\pi_{\xi}(f_{\chi})\| \leq C(1+\|\xi\|^2)^N  e^{-\beta \pi \|\xi\|^2}$. 
\end{proof}
\noindent Let $G^c$ denote the maximal compact subgroup of $G$. Then $G^c$ is connected, contained in $Z(G)$ and $G/G^c$ is simply connected.

\begin{lem}\label{quotient}
Let $G$ be a connected nilpotent Lie group. Suppose that the Theorem \ref{simp} holds for all quotient subgroups $H = G/C,$ where $C$ is a closed subgroup of $G^{c}=Z(G)^{c}$ such that $Z(G)^{c}=C$ or $Z(G)^{c}/C = \mathbb{T}$. Then Theorem \ref{simp} also holds for $G$.
\end{lem}
\begin{proof}
Let $K = Z(G)^{c}$ and $f : G \to \mathbb{C}$ be a measurable function that satisfies the  conditions of Theorem \ref{simp}. For $\chi$  in $\widehat{K}$, consider $ K_{\chi} = \{k\in K: \chi(k)=1\}$ and $H = G/K_{\chi}$. Then $f_{\chi}$ is constant on the cosets of the subgroup $K_{\chi}$ and also by Lemma \ref{char}, it follows that the function $ f_{\chi}$ satisfies the Hardy's type decay conditions. Since $H^{c} = K/K_{\chi} = \mathbb{T}$ or $H^c = \{e\}$, therefore on using the hypothesis we get $f_{\chi} = 0$ a.e. As  $\chi \in \widehat{K}$ is arbitrary chosen, therefore we have $f = 0$ a.e.
\end{proof}

\noindent For a second countable, locally compact group $G$ containing $\mathbb{R}$ as a closed central subgroup, let $S$ denote a Borel cross-section for the cosets of $\mathbb{R}$ in $G$. The inverse image of Haar measure on $G/\mathbb{R}$ under the map $s\to \mathbb{R}s$ from $S\to G/\mathbb{R}$ is denoted by $ds$. 
\begin{lem}\label{bound}

Let $G$ and $S$ be as defined above and  $f : G \to \mathbb{C}$ be  a measurable function  satisfying 
$$|f(ts)| \leq (1+|t|^2)^N e^{-\alpha \pi t^2}\phi(s),$$  for some $\alpha > 0$ and $\phi \in L^2(S)$. Define a function $g$ on $\mathbb{R}$ such that 
 $g(t)= \int_{S}(f_s\ast f_s^{\ast})(t)\ ds$. Then $$|g(t)| \leq C_{1}  e^{-\gamma\pi\frac{t^2}{2}},$$ for some $C_1>0$ and $0 < \gamma < \alpha.$ 

\end{lem}
\begin{proof}
For each $t\in \mathbb{R}$ and $0< \gamma< \alpha$, we have \begin{align*}
|g(t)| &= |\int_{S}\int_{R}f(zs)\ \overline{f((z-t)s)}\ dz\ ds|\\
& \leq \int_{S}\int_{\mathbb{R}}|f(zs)|\ |f((z-t)s)|\ dz\ ds\\
&\leq \int_{S}\phi(s)^2 ds \int_{\mathbb{R}}(1+|z|^2)^N
(1+|z-t|^2)^N e^{-\pi\alpha( z^2 +(z-t)^2)}\ dz\\
&\leq \|\phi\|_2^2\ \int_{\mathbb{R}}\sum\limits_{k=0}^N \sum\limits_{j=0}^N   {{N}\choose{k}}{{N}\choose{j}} z^{2k}(z-t)^{2j}e^{-(\alpha -\gamma)\pi z^2} e^{-\gamma\pi z^2} e^{-(\alpha - \gamma)\pi(z-t)^2}e^{-\gamma\pi(z-t)^2}dz.
\end{align*}
The function  $z\to {{N}\choose{k}} z^{2k}e^{-(\alpha-\gamma)\pi z^2}$ is bounded on $\mathbb{R}$ say by $K_{k}$. Set $K = \max\{K_{k}: 0\leq k\leq N\}$. Thus,
  it follows that
\begin{align*}
|g(t)| \leq K (N+1)\ \|\phi\|_{2}^2\  \sum\limits_{j=0}^N {{N}\choose{j}} \int_{\mathbb{R}}(z-t)^{2j}e^{-\gamma \pi z^2}e^{-(\alpha-\gamma)\pi(z-t)^2}e^{-\gamma \pi (z-t)^2}dz.
\end{align*}
Using Cauchy-Schwarz inequality, we have
\begin{align*}
|g(t)|& \leq K(N+1)\ \|\phi\|_{2}^2 \  \sum\limits_{j=0}^N {{N}\choose{j}} \left(\int_{\mathbb{R}}(z-t)^{4j}e^{-2(\alpha-\gamma)\pi(z-t)^2}dz\right)^{1/2}\left(\int_{\mathbb{R}}e^{-2\gamma\pi z^2}e^{-2\gamma \pi (z-t)^2}dz\right)^{1/2}\\
&=  K(N+1)\ \|\phi\|_{2}^2\  \sum\limits_{j=0}^N {{N}\choose{j}} B_j\left(\int_{\mathbb{R}}e^{-2\gamma\pi( \frac{t^2}{2} +\frac{1}{2}(2z-t)^2)}dz\right)^{1/2}\\
&=K (N+1)\ \|\phi\|_{2}^2\  e^{-\gamma\pi\frac{t^2}{2}}   \sum\limits_{j=0}^N {{N}\choose{j}} B_j \int_{\mathbb{R}}e^{-\pi\gamma\frac{1}{2}(2z-t)^2)}dz \\
&=K (N+1)\ \|\phi\|_{2}^2\  e^{-\gamma\pi\frac{t^2}{2}}   \sum\limits_{j=0}^N {{N}\choose{j}} B_j\int_{\mathbb{R}}e^{-2\pi\gamma z^2}dz  \\
&= \frac{1}{\sqrt{2\gamma}}K (N+1)\ \|\phi\|_{2}^2\  e^{-\gamma\pi\frac{t^2}{2}}   \sum\limits_{j=0}^N {{N}\choose{j}} B_j\\
&= C_{1}  e^{-\gamma\pi\frac{t^2}{2}},
\end{align*}
where $C_{1}= \frac{1}{\sqrt{2\gamma}} K (N+1)\ \|\phi\|_{2}^2\ \sum\limits_{j=0}^N {{N}\choose{j}} B_j $ 
and $ B_{j}=\left(\int_{\mathbb{R}}(z-t)^{4j}e^{-2(\alpha-\gamma)\pi(z-t)^2}dz\right)^{\frac{1}{2}}$. \hspace{1cm} \qedhere 
\end{proof}
\noindent We shall now prove Hardy's type theorem for Fourier transform for  connected nilpotent Lie groups  having non-compact center.
\noindent Consider $V_{k}= [\xi_1-1/2k, \xi_1+1/2k] $ for every natural number $k $ and fix real number $\xi_1$.  For $m>2k$ choose a  $C^{\infty}$ function $v_{k,m}$ on real line  such that support of $v_{k,m}$ is contained in  $V_{k}$,  $v_{k,m}=1$ on $[\xi_1-1/2k+1/m, \xi_1+1/2k-1/m]$ and $0\leq v_{k,m}\leq 1$.
By Plancherel inversion theorem there exists  $u_{k,m}\in L^1(\mathbb{R})$  such that $\widehat{u_{k,m}}=v_{k,m}.$ For $f\in L^1(G),$ consider $f_{k,m}= u_{k,m}\ast f$ and define $F_{k,m} :G \to \mathbb{C}$ by 
$$F_{k,m}(x)= \int_{\mathbb{T}}(f_{k,m}\ast f_{k,m}^{\ast})(xz)\ dz, x\in G.$$
 Next, we modify the Lemma 3.1 proved in \cite{Bak:Kan:08} in order to prove Theorem \ref{simp}. 
\begin{lem}
Let $f:G\to \mathbb{C}$ be a measurable function satisfying condition (i) of Theorem \ref{simp}. Then
$$\lim_{k,m\to \infty}\ kF_{k,m}(e)= 0.$$\end{lem}
\begin{proof}
For fix $z,w \in \mathbb{T}$ and $y\in Y$, define
$$E_{k,m}(z,w,y)= \int_{\mathbb{R}}f(t,z,y)\left(\int_{\mathbb{R}}u_{k,m}(s) \overline{(u_{k,m}\ast f)(t+s,w,y)}ds\right)dt.  $$
Then as proved in \cite[Lemma 3.1]{Bak:Kan:08}, we have
\begin{align}\label{mod eq}
F_{k,m}(e)=\int_{Y}\int_{\mathbb{T}^2} E_{k,m}(z,w,y)dz\ dw\ dy
\end{align}
and 
\begin{align*}
E_{k}(z,w,y)&=\lim_{m\to\infty}E_{k,m}(z,w,y)\\
&=\int_{\mathbb{R}}f(t,z,y)\int_{\xi_1-1/2k}^{\xi_1+1/2k}\widehat{u_{k,m}}(s)\widehat{u_{k,m}}(t,s)\overline{\widehat{f}(t+s,w,y)}ds\ dt.
\end{align*}
Now $1_{V_k}(t+s)=0$ for all $s\in [\xi_1-1/2k,\xi_1+1/2k]$ whenever $t\notin [-1/k, 1/k]$ and if $t\in [-1/k,1/k]$ then 
$$1_{V_k}(t+\cdot)=1_{[\xi_1-t-1/2k,\xi_1-t+1/2k]}\leq 1_{[\xi_1-3/2k,\xi_1+3/2k]}.$$
Using condition (1.1) of Theorem \ref{simp}, we compute 
\begin{align}\label{mod eq1}
|E_{k}(z,w,y)|&\leq \int_{-1/k}^{1/k}|f(t,z,y)|\left(\int_{\xi_1-3/2k}^{\xi_1+3/2k}|\widehat{f}(t+s,w,y)|ds \right)dt \nonumber \\
&\leq \frac{3}{k}\|\widehat{f}\|_{\infty} \int_{-1/k}^{1/k}|f(t,z,y)|dt \nonumber \\
&\leq \frac{ 3C}{k}\|\widehat{f}\|_{\infty} \phi(y) \int_{-1/k}^{1/k}(1+t^2)^N e^{-\alpha \pi t^2}\ dt \nonumber\\
&\leq \frac{3C}{k^2} 2^{(N+1)} \|\widehat{f}\|_{\infty}\ \phi(y) .
\end{align}
Therefore, from (\ref{mod eq}) and (\ref{mod eq1}), it follows that 
\begin{align*}
\lim_{m\to \infty}|F_{k,m}(e)|&\leq \int_{Y}\int_{\mathbb{T}^2}|E_{k}(z,w,y)| dz\ dw\ dy\\
&\leq \frac{3C}{k^2} 2^{(N+1)}\|\widehat{f}\|_{\infty}\int_{Y}\varphi(y)\ dy. 
\end{align*}
Hence, $ \lim\limits_{k,m\to\infty}F_{k,m}(e)=0$.
\end{proof}
\noindent It may be observed that the proof of Theorem  \ref{simp}  now follows from the technique used in    \cite[Theorem 1.1]{Bak:Kan:08}. But, for the sake of completeness, we briefly sketch the proof.
 For fix $\xi_2 \in \mathbb{R}$, from  \cite{Bak:Kan:08}, we have
 \begin{align*}
 \widehat{g}(\xi_2) = \lim_{k\to \infty}\int_{V_k}\left(\int_{X_{\eta_2}}|Pf(\eta)|\cdot \|\pi_{\eta}(f)\|_{\text{HS}}^2\ d\eta' \right)
 \end{align*}
and 
\begin{align*}
&\int_{X_{\eta_2}}|Pf(\eta)|\cdot \|\pi_{\eta}(f)\|_{\text{HS}}^2\ d\eta'\\
&\leq C\sum_{n\in \mathbb{Z}^*}\left( \int_{V_T{''}}|Pf(\eta)| (1+\|\eta\|^2)^N exp(-2\beta(n^2+\eta_2^2+\|\eta''\|^2))d\eta''\right)
\end{align*}
where $V_T'' = \sum_{i\in T, i>2}\mathbb{R}X_{i}^*$. Let $0 < \delta < \beta $. Since $Pf$ is a polynomial function  in $\eta$, therefore there exist a constant $K >0$ such that  for all $\eta \in \W$
\begin{align*}
|Pf(\eta)|(1+\|\eta\|^2)^N \exp(-2(\beta - \delta)\|\eta\|^2) \leq K.
\end{align*} 
As proved in \cite{Bak:Kan:08}, we have 
\begin{align*}
|\widehat{g}(\xi_2)| \leq D \exp(-2\delta\xi_2^2)
\end{align*}
for all $\xi_2 \in \mathbb{R}$ and $D>0$.
 By  Lemma \ref{bound}, for all $t\in \mathbb{R}$, we have
 \begin{align*}
 |g(t)| \leq C_1 e^{-\gamma t^2/2}
\end{align*}   for some $C_1>0$ and $0<\gamma < \alpha$. Since $\alpha\beta>1$,  we can choose  $\gamma $ and $\delta$ such that $\gamma\delta>1.$ Then by Hardy's theorem for $\mathbb{R}$, we get $g = 0$ a.e. But, $g$ is integral of a positive definite function $f_s\ast f^\ast_s$ on $\mathbb{R}$ which imply that $f=0$ a.e.\\

We conclude this section by remarking, if $G$ is a connected nilpotent Lie group that has no square integrable irreducible representation and all the co-adjoint orbits in $\g^{\ast}$ are flat, then Hardy's type theorem holds for $G$. Let $K$ be any compact central subgroup of $G$. Then $H = G/K$ has no square integrable irreducible representation and also satisfies flat orbit condition. By Lemma \ref{quotient}, it is enough to prove Hardy's type theorem for such group $H$ satisfying $H^c = \mathbb{T}$. But, then $H$ must have a non-compact centre and by Theorem \ref{simp}, $H$ satisfies Hardy's type theorem. 
Also in view of \cite[Proposition 4.1]{Bak:Kan:08}, it is easy to see that Theorem \ref{simp} does not hold for nilpotent Lie groups having an irreducible square integrable representation in particular reduced Weyl-Heisenberg group, low dimensional nilpotent Lie groups $G_{5,1}/\mathbb{Z}$, $G_{5,3}/\mathbb{Z}$ and $G_{5,6}/\mathbb{Z}$. For more deatils of such groups, one may refer to \cite{Niel:83}.

\section{Analogue of Hardy's theorem for Gabor transform}
\noindent In this section,  we deal with an analogue of Hardy's theorem  for Gabor transform.
\begin{lem}\label{zero}
Let $G$ be a second countable locally compact group. For $f,\psi \in L^2(G)$ and $x\in G$, define  $f_{\psi}^{x} : G \to \mathbb{C}$ such that $$f_{\psi}^{x}(y)=f(y)\ \overline{\psi(x^{-1}y)}.$$
If $f_{\psi}^{x} =0 $ a.e. for almost all $x\in G$, then either $f=0$ a.e. or $\psi = 0$ a.e.
\end{lem}
\begin{proof}Let us assume that  $\psi $ is a non-zero function in $L^2(G)$.
 There exist a zero subset $M$ of $G$ such that for all $x\in G \setminus M$, $f_{\psi}^{x} =0 $ a.e. But, $G\setminus M$ is dense in $G$ and $G$ is second countable, so we can take a sequence $(x_j)_{j\in \mathbb{N}}$ contained in $G\setminus M,$ which is dense in $G.$ Let 
\begin{align*}
V = \left\{t \in G: |\psi(t)|> \frac{1}{2||\psi||_{\infty}}\right\}.
\end{align*}
Then $V$ is a non-empty open subset of $G$ and $\bigcup\limits_{j\in \mathbb{N}}x_{j}V = G.$ Consider the function  
\begin{align*}
h(t) = \sum\limits_{j\in \mathbb{N}}\frac{1}{2^j}|\psi(x_j^{-1}t)|,\  t\in G.
\end{align*}
Clearly $h$ is a strictly positive function on $ G $. Moreover,
\begin{align*}
0 \leq \int_{G}|f(t)|h(t)\ dt 
&= \int_{G}\sum\limits_{j\in \mathbb{N}}\frac{1}{2^j}|f(t)||\psi(x_j^{-1}t)|\ dt
&= \sum\limits_{j\in \mathbb{N}}\frac{1}{2^j}\int_{G}|f_{x_j}^{\psi}(t)|\ dt = 0.
\end{align*}
Hence, $\int_{G}|f(t)|h(t)\ dt = 0$ which implies that $f\cdot h = 0 $ a.e. Since $h$ is strictly positive, therefore it follows that $f = 0$ a.e.
\end{proof}

\begin{thm}
Let $ f$   be a measurable function on $\mathbb{R}^n$ such that 
$|f(x)|\leq Ce^{-\alpha \pi \|x\|^2}$ \text{ for all } $x \in \mathbb{R}^n $ and $\psi$ be a window function.
Also assume that for almost all $y\in \mathbb{R}^n$,
$$|G_{\psi}f(y,\xi)|\leq \eta_{y}\ e^{-\beta\pi\|\xi\|^2} \text{ for all }  \xi \in \mathbb{R}^n, $$  
where $\alpha, \beta,C$ and $\eta_{y}$ are positive scalers and $\eta_{y}$  depends upon $y$.\\
 If $\alpha\beta > 1$, then either $ f = 0 $ a.e. or $\psi = 0$ a.e.
\end{thm}
\begin{proof}
For each $y\in \mathbb{R}^n$,  define the function $F_{y}:\mathbb{R}^n \to \mathbb{C}$ such that  $$F_{y}(x)= f_{\psi}^{y}\ast (f_{\psi}^{y})^{\ast}(x).$$ Then for each $\xi \in \mathbb{R}^n$, we have $$\widehat{F_{y}}(\xi)  = |\widehat{f_{\psi}^{y}(\xi)}|^2 = |G_{\psi}f(y,\psi)|\leq \eta_{y}^2\  e^{-2\beta \pi \|\xi\|^2}.$$
Also, for each $x\in \mathbb{R}^n,$ we obtain  \begin{align*}
|F_{y}(x)|& \leq \int_{\mathbb{R}^n}|f_{\psi}^{y}(t)|\ |f_{\psi}^{y}(t-x)|\ dt\\
&= \int_{\mathbb{R}^n}|f(t)|\ |\psi(t-y)|\ |f(t-x)|\ |\psi(t-x-y)|\ dt\\
&= \int_{\mathbb{R}^n} C^2\ e^{-\alpha \pi \|t\|^2}e^{-\alpha \pi \|t-x\|^2}|\psi(t-y)|\ |\psi(t-x-y)|\ dt\\
&= C^2\int_{\mathbb{R}^n} e^{-\alpha \pi (\frac{\|x\|^2}{2}+\frac{1}{2}(\|2t-x\|^2))}|\psi(t-y)|\ |\psi(t-y-x)|\ dt \\
&\leq C^2\ e^{-\alpha\pi\frac{ \|x\|^2}{2}} \int_{\mathbb{R}^n}|\psi(t-y)||\psi(t-y-x)|\ dt\\
&\leq  C^2\ e^{-\alpha\pi \frac{\|x\|^2}{2}}(|\psi|\ast|\psi|^{\ast})(x)\\
&\leq  C^2\ e^{-\alpha\pi \frac{\|x\|^2}{2}}\|\ |\psi|\ast|\psi|^{\ast}\|_{\infty}.
\end{align*}
Taking $C_1=\max \{ \eta_{y}^2,\ C^2\ \|\ |\psi|\ast|\psi|^{\ast}\|_{\infty}\}.$ Then,
$$|F_{y}(x)| \leq C_1e^{-\alpha\pi\frac{\|x\|^2}{2}}\text{ for all } x\in \mathbb{R}^n$$ and 
$$|\widehat{F_{y}(\xi)}|\leq C_1e^{-2\beta\pi\|\xi\|^2} \text{ for all } \xi \in \mathbb{R}^n.  $$ Using Hardy's theorem for $\mathbb{R}^n$, it follows that $F_y = 0 $ for almost all $y\in \mathbb{R}^n$ which further implies that $f_{\psi}^{y} = 0 $ for almost all $y \in \mathbb{R}^n.$ Therefore, from  using Lemma \ref{zero}, either $f=0$ a.e. or $\psi = 0 $ a.e.
\end{proof}

\begin{thm}
Let $G$ be a connected and simply connected nilpotent Lie group with non-compact centre. Suppose that $\psi \in C_{c}(G)$ and $f\in L^2(G)$ satisfies  
$$\|G_{\psi}f(x,\pi_{\xi})\|_{\text{HS}}\leq C_{x}\ e^{-\pi\beta\|\xi\|^2},$$
where $C_{x}$ is a positive scalar depending on $x.$ 
If $\beta>0$, then either $f=0$ a.e. or $\psi =0$ a.e. 
\end{thm}
\begin{proof}
For $y=(y_2,y_3,\dots,y_{n})\in \mathbb{R}^{n-1}$, define a function $f_{y}:\mathbb{R} \to \mathbb{C}$ such that $$f_{y}(x_1)= f(\exp(x_1X_1+\sum\limits_{j=2}^{n}y_{j}X_j)).$$
For $z\in G$, define a function $F_{z}:\mathbb{R} \to \mathbb{C}$ given by
 $$F_{z}(x_1) = \int_{\mathbb{R}^{n-1}}(f^{z}_{\psi})_{y}\ast(f^{z}_{\psi})_{y}^{\ast}\ dy. $$
\\
As $\psi \in C_{c}(G),$  therefore $f^{z}_{\psi}$ has compact support.
Moreover, \begin{align*}
F_{z}(x_1)& = \int_{\mathbb{R}^{n-1}}(f^{z}_{\psi})_{y}\ast(f^{z}_{\psi})_{y}^{\ast}\ dy \\
&= \int_{\mathbb{R}^{n-1}}\int_{\mathbb{R}}f^{z}_{\psi}(t,y)\overline{f^{z}_{\psi}(t-x_1,y)}\ dy\ dt\\
&=f^{z}_{\psi}\ast f^{z}_{\psi}(x_1,e_1).
\end{align*}
Therefore, $F_{z}$ is a continuous function with compact support say $K$. Choose $\alpha > 0$ such that $\alpha \beta >1.$ Since the function $x_1 \to\exp(-\alpha\pi x_1^2) $ attains minima on $K$, therefore $r\leq e^{-\pi \alpha x_1^2}$ for some $r> 0$. Also, there exists $C_1>0$ such that $|F_z(x_1)|\leq C_1,$ for all $x_1 \in \mathbb{R}.$ Choose $C^{'}>0$ satisfying $rC^{'} > C_1$ and therefore for each $x\in K$, we obtain
$$|F_z(x_1)|\leq C_1< rC^{'}< C^{'}e^{-\pi \alpha x_1^2},$$
and for $x_1\in \mathbb{R}\setminus K,$ we have $ F_z(x_1) = 0$.
Also $f^{z}_{\psi} \in L^1\cap L^2(G) $ and $$\|\pi_{\xi}(f^{z}_{\psi})\|_{\text{HS}}\leq \|G_{\psi}f(x,\pi_{\xi})\|_{\text{HS}} \leq Ce^{-\pi\beta \|\xi\|^2}.$$
 Using \cite[Lemma 2]{Kan:Kum:01}, we get that $|\widehat{F_z}(\xi_1)|\leq c\ e^{-2\pi\beta \|\xi\|^2}$, for some $c >0$ . Therefore, using Hardy's theorem for Fourier transform, the function $F_z = 0 $ a.e.  Since $F_z$ is integral of a positive definite function $(f^{z}_{\psi})_{y}\ast(f^{z}_{\psi})_{y}^{\ast}$, therefore $(f^{z}_{\psi})_{y} = 0$ a.e. This holds for all $z\in G$  which further gives that either $f=0$ a.e. or $\psi =0$ a.e.
\end{proof}
\noindent  The next result directly follows from the above theorem.
\begin{thm}
Let $G$ be a connected and simply connected nilpotent Lie group. Let $\psi\in C_c(G)$ and $f\in L^2(G)$ such that 
\begin{align*}
\|G_{\psi}f(x,\pi_{\xi})\|\leq Ce^{-\pi(a\|x\|^2+b\|\xi\|^2)/2}
\end{align*}
for all $(x,\xi) \in  G \times \mathcal{W}$, where $a,b$ and $C$ are positive real numbers. Then, either $f=0$ a.e. or $\psi = 0$ a.e.
\end{thm}
\section{ Beurling Theorem}

\noindent The Beurling theorem for Gabor transform on connected nilpotent Lie group $ G $ can be stated as follows:\\
\textbf{Beurling Theorem}: Let $ f $ and $ \psi $ are square integrable functions on $ G $ such that \[
\int_{G}\int_{\mathcal{W}} \|G_{\psi}f(x, \pi_{\xi})\|_{\text{HS}}e^{\pi (\|x\|^2+\|\xi\|^2)/2}\ Pf(\xi)\ d\xi\  dx < \infty.  \]
Then either $ f=0 $ a.e. or $ \psi=0 $ a.e.\\

\noindent In the next theorem, we  partially prove the above result.
\begin{thm}  Let $\psi\in C_{c}(G)$ and $f \in L^2(G)$, $G$ be a connected and simply connected nilpotent Lie group, such that \begin{align}
\int_{G}\int_{\W} \|G_{\psi}f(x,\pi_{\xi})\|_{\text{HS}}\ e^{\pi(\|x\|^2+\|{\xi}\|^2)} Pf(\xi)\ dx \ d\xi < \infty \label{nil-ber}.
\end{align}
Then  either $f=0$ a.e. or $\psi = 0$ a.e.
\begin{proof} From (\ref{nil-ber}), there exist a zero set $M \subset G$ such that for all $ x\in G\setminus M$ we have
\begin{align}
\int_{\widehat{G}} \|G_{\psi}f(x,\pi_{\xi})\|_{\text{HS}}\ e^{\pi(\|x\|^2+\|\xi\|^2)}  Pf(\xi)\ d\xi < \infty.\label{int:gab}
\end{align}
For $x\in G\setminus M$, we consider the function $f^{x}_{\psi}$
and compute \begin{align*}
&\int_{G}\int_{\W}|f^{x}_{\psi}(z)|\ \|\widehat{f^{x}_{\psi}(\pi_{\xi})}\|_{\text{HS}}\ e^{2\pi \|z\|\|\xi\|} Pf(\xi) \ dz\ d\xi \\
&\leq 
\int_{G}\int_{\W}|f^{x}_{\psi}(z)|\ \|\widehat{f^{x}_{\psi}(\pi_{\xi})}\|_{\text{HS}}\ e^{\pi(\|z\|^2+\|\xi\|^2)} Pf(\xi) \ dz\ d\xi \\
&= \int_{G}\int_{\W}|f^{x}_{\psi}(z)|\ \|G_{\psi}f(x,\pi_{\xi})\|_{\text{HS}}\ e^{\pi(\|z\|^2+\|{\xi}\|^2)} Pf(\xi) \ dz\ d\xi\\
&= \int_{G}|f^{x}_{\psi}(z)| e^{\|z\|^2}dz \int_{\W}   \|G_{\psi}f(x,\pi_{\xi})\|_{\text{HS}}\  e^{\pi\|\xi\|^2} Pf(\xi)\  d\xi.
\end{align*}
Also, \begin{flalign}
 &&\int_{G}|f^{x}_{\psi}(z)| e^{\pi\|z\|^2}dz &= \int_{G} |f(z)||\psi(x^{-1}z)| e^{\pi\|z\|^2}dz & \nonumber \\
&&& \leq  \left( \int_{G} |f(z)|^2 dz\right)^{1/2}  \left(\int_{G} |\psi(x^{-1}z)|e^{2 \pi \|z\|^2}dz\right)^{1/2}\label{sq:int}.
\end{flalign}
Since $\psi \in C_{c}(G),$ therefore $\psi \cdot e^{\pi\|\cdot\|^2} \in L^2(G)$ and $\int_{G}|f^{x}_{\psi}(z)| e^{\pi \|z\|^2}dz < \infty$. Thus, using (\ref{int:gab}) and (\ref{sq:int}), we get 
$$\int_{G}\int_{\W}|f^{x}_{\psi}(z)|\ \|\widehat{f^{x}_{\psi}(\pi)}\|_{\text{HS}}\ e^{2\pi \|x\|\cdot\|\pi\|} Pf(\xi)\ dx\ d\xi < \infty.$$
Using Beurling theorem for simply connected nilpotent Lie groups \cite{Sma1:11}, it follows that $f^{x}_{\psi} = 0$ a.e. for all $x\in G\setminus M$. Hence, by Lemma \ref{zero},  either $f = 0$ a.e. or $\psi =0$ a.e.  
\end{proof}
\end{thm} 

\begin{rem}
Let $ G $ be a connected nilpotent Lie group with a square integrable representation. Then as proved in \cite[Theorem 5.1]{ash:18}, there exist non-zero functions $ f $ and $ \psi $ in $ L^2(G) $ such that for all $x\in G$ and ${\xi}\in \mathcal{W}$,
\begin{align*}
\|G_{\psi}f(x, \pi_{\xi})\|_\text{HS} \leq C e^{-\pi (a \|x\|^2 +b \|\xi\|^2)/2}
\end{align*}  
 where $ a,b $ are non-negative real numbers with $ab>1$ and $ C$ is a positive constant.
 For $a,b >1$,  it follows that 
 \begin{align*}
 \int_{G}\int_{\mathcal{W}} \|G_{\psi}f(x, \pi_{\xi})\|_{\text{HS}}\ e^{-\pi (\|x\|^2+\|\xi\|^2)/2}\ Pf(\xi)\ d\xi\ dx < \infty.
\end{align*} 
Thus, Beurling theorem does not holds for G. Several examples of such type of group exist including Weyl-Heisenberg group, low dimensional nilpotent Lie groups $G_{5,1}/\mathbb{Z}$, $G_{5,3}/\mathbb{Z}$ and $G_{5,6}/\mathbb{Z}$. One can create more such examples using the following:
\end{rem}
 
\begin{prop}
Let $G$ be a group of the form $G = A\times K\times D$, where $A$ is a nilpotent Lie group, $K$ is compact group and $D$ is type I discrete group.  If Beurling theorem fails for  $A$, then it also fails for $G$.
\end{prop}
\begin{proof}

Since Beurling theorem fails for $A$, therefore there exist non-zero functions $f,\psi\in L^2(A)$ such that 
\begin{align*}
\int_{A}\int_{\mathcal{W}}\|G_{\psi}f(x,\pi_{\xi})\|_{\text{HS}}\ e^{\pi(\|x\|^2+\|\xi\|^2)/2} Pf(\xi)\ dx\ d\xi < \infty.
\end{align*}
Define the functions $F,\Psi: G\to \mathbb{C}$ by 
$$F(x,k,t)= f(x)\chi_{e}(t)
 \quad \quad  \text{and} \quad \quad
  \Psi(x,k,t) = \psi(x) \chi_{e}(t),$$
  where $e$ being the identity of $D$. Let $\{e_i^{\xi}\}, \{e_{i}^{\delta}\}$ and $\{e_{i}^{\gamma}\}$  be  orthonormal basis  of Hilbert spaces  corresponding to the representations $\pi_{\xi}, \delta$ and $\gamma$ of $A, K$ and $D$ respectively. Then,
\begin{flalign*}
&\langle G_{\Psi}F(x,k,t,\pi_{\xi}, \delta, \gamma)e_{i}^{\xi}\otimes e_{m}^{\delta} \otimes e_{p}^{\gamma}, e_{j}^{\xi}\otimes e_{n}^{\delta} \otimes e_{q}^{\gamma}\rangle &\\
 & \qquad \qquad \qquad \qquad = \begin{cases}
\langle G_{\psi}f(x,\pi_{\xi})e_{i}^{\xi}, e_{j}^{\xi}\rangle & \text{ if } t = e \text{ and } \delta \equiv I\\
0, & \text{ otherwise}.
\end{cases}
\end{flalign*}
  Also, using \cite{Moore:72, Palm78}, $D$ is bounded dimensional representation group. So, there exists a positive scaler $M$ such that dim$(\gamma)\leq M $ for all $\gamma\in \widehat{D}.$  Therefore, we have
\begin{flalign*}
&\| G_{\Psi}F(x,k,e,\pi_{\xi}, I, \gamma)\|^2_{\text{HS}} &\\
&\quad \leq \sum_{i,j}\sum_{m,n}\sum_{p,q}| \langle G_{\Psi}F(x,k,e,\pi_{\xi}, I, \gamma)e_{i}^{\xi}\otimes e_{m}^{\delta} \otimes e_{p}^{\gamma}, e_{j}^{\xi}\otimes e_{n}^{\delta} \otimes e_{q}^{\gamma}\rangle|^2 &\\
&\quad =\sum_{i,j}\sum_{m,n}\sum_{p,q}|\langle G_{\psi}f(x,\pi_{\xi})e_{i}^{\xi}, e_{j}^{\xi}\rangle|^2
\leq M^2
 \|G_{\psi}f(x,\pi_{\xi})\|^2_{\text{HS}}.
\end{flalign*}
Thus,
\begin{align*}
&\int_{A}\int_{K}\sum_{t\in D} \int_{\mathcal{W}}\sum_{\delta \in \widehat{K}} \int_{\widehat{D}}\|G_{\psi}f(x,k,t,\pi_{\xi},\delta,\gamma)\|_{\text{HS}}\ e^{\pi(\|x\|^2+\|\xi\|^2)/2} Pf(\xi) dx\ dk\ d\xi\ d\gamma\\
& \leq \int_{A}\int_{K} \int_{\mathcal{W}} \int_{\widehat{D}}\|G_{\psi}f(x,k,e,\pi_{\xi},I,\gamma)\|_{\text{HS}}\ e^{\pi(\|x\|^2+\|\xi\|^2)/2} Pf(\xi) dx\ dk\ d\xi\ d\gamma\\
&= \int_{A}\int_{\mathcal{W}}  \|G_{\psi}f(x,\pi_{\xi})\|_{\text{HS}}\ e^{\pi(\|x\|^2+\|\xi\|^2)/2} Pf(\xi) dx\  d\xi\ < \infty.
\end{align*}
Hence, Beurling theorem fails for $G$.
\end{proof}
Next we look at an analogue of Beurling's theorem for Fourier transform on abelian groups. We could not find a reference for this result, so a proof has been included. Let $G$  be a second countable, locally compact, abelian group with  dual group  $\widehat{G}$. For $z\in G$ and $\omega\in \widehat{G}$, we define the \textit{translation operator} $T_z$ on $L^2(G)$ as
\begin{flalign*}
(T_z f)(y)=f(z^{-1}y)
\end{flalign*}
and the \textit{modulation operator} $M_\omega$ on $L^2(G)$ as
\begin{flalign*}
(M_\omega f)(y)=f(y)\ \omega(y),
\end{flalign*}
where $f \in L^2(G)$ and $y\in G$. For $f,\psi \in L^2(G)$, the following  property of the Gabor transform can be easily verified:
\begin{flalign}
\label{eq1} &&&G_\psi (M_\omega T_z f)(x,\gamma)=(\omega^{-1}\gamma)(z^{-1})\ G_\psi f(z^{-1}x,\omega^{-1}\gamma)&
 \end{flalign}
 for all $ x,z\in G \text{ and } \gamma,\omega\in \widehat{G}.$ 

Using structure theory of abelian groups \cite{Hew:Ros:63:70}, $G$ decomposes into a direct product $G=\mathbb{R}^n \times S$, where $n\geq 0$ and $S$ contains a compact open subgroup. So, the connected component of identity of $G$ in non-compact if and only if $n\geq 1$. Let $G=\mathbb{R}^n \times S$ has non-compact connected component of identity. The dual group $\widehat{G}$ is identified with $\widehat{G} = \widehat{\mathbb{R}^n}\times \widehat{S}$.
 \begin{thm}\label{ber}
  Let $f \in L^1\cap L^2(\mathbb{R}^n \times S)$ such that 
  $$\int_{\mathbb{R}}\int_{S} \int_{\mathbb{R}^n} \int_{\widehat{S}} |f(x,s)||\widehat{f}(\xi, \gamma)| e^{2\pi{|x\cdot\xi|}}\ dx\ ds\ d\xi\ d\gamma< \infty.$$
  Then $f = 0$ a.e.
\end{thm}

 \noindent Before proving the above theorem, we shall prove some lemmas.
\begin{lem}\label{zeroall}
Let $f\in L^2(\mathbb{R}^n\times K)$, where $K$ is a compact group not necessarily abelian. For $\gamma \in \widehat{K},$ let $\mathcal{H}_{\gamma}$ be the Hilbert space of dimension $d_{\gamma}$ with orthonormal basis $\{e_{i}^{\gamma}\}_{i=1}^{d_{\gamma}}$. For fixed $e_{i}^{\gamma}$ and $e_{j}^{\gamma}$, define $f_{\gamma}:\mathbb{R}^n \to \mathbb{C}$ such that
\begin{flalign*}
f_{\gamma}(x)=\int_{K}{f(x,k)\ \overline{\langle{\gamma(k)^\ast e_{i}^{\gamma},e_{j}^{\gamma}}\rangle}}\ dk.
\end{flalign*}
 If for each $\gamma \in \widehat{K}$ and for all $i,j$ from $1 \text{ to } d_{\gamma}$, the function $f_{\gamma}=0 $ a.e., then $f=0$ a.e.
\end{lem}
\begin{proof}
For $\gamma  \equiv 1$,  $f_{\gamma}= 0$ a.e. implies
\begin{flalign*}
\int_{\mathbb{R}^n}\int_{K} f(x,k)\ dx\ dk = 0. 
\end{flalign*}
Thus, $f$ is an integrable function. For fixed $\gamma \in \widehat{K}$ and $\xi\in \mathbb{R}^n$, we obtain
\begin{flalign*}
\langle \xi\otimes \gamma(f)e_{i}^{\gamma},e_{j}^{\gamma}\rangle= \int_{\mathbb{R}^n}\int_{K} f(x,k) e^{-2\pi ix\cdot\xi}\  \overline{\langle{\gamma(k)^\ast e_{i}^{\gamma},e_{j}^{\gamma}}\rangle}\ dx\ dk = 0.
\end{flalign*}
Since $\gamma\in \widehat{K} \text{ and } \xi \in \mathbb{R}^n$ are arbitrarily fixed, therefore $\langle \xi\otimes \gamma(f)e_{i}^{\gamma},e_{j}^{\gamma}\rangle = 0
$  for all $\gamma$ and $\xi$.
But, $f \in L^1\cap L^2(G)$, therefore  using (\ref{planc}), we conclude that $f=0$ a.e.
\end{proof}

\begin{lem}\label{comp2}
Let $f\in L^1\cap L^2(\mathbb{R}^n\times K)$, where $K$ is a compact group  satisfying 
$$\int_{\mathbb{R}^n} \int_{K} \int_{\mathbb{R}^n}\int_{\widehat{K}}|f(x,s)|\ \|\xi\otimes \gamma(f)\|_{\text{HS}}\ e^{2\pi|x\cdot\gamma|}\ dx\ d\xi\ ds\ d\gamma < \infty.$$ Then $f=0$ a.e.
\end{lem}
\begin{proof}
For $\gamma \in \widehat{K}$, let $f_{\gamma}$ be as in Lemma \ref{zeroall}.  
 For $\xi \in \mathbb{R}^n$, we obtain 
 \begin{align*}
| \widehat{f_{\gamma}}(\xi)|  &= |\langle \xi \otimes \gamma (f)e_{i}^{\gamma}, e_j^{\gamma}\rangle|
\leq \| \xi \otimes \gamma (f)\|_{\text{HS}}.
 \end{align*}
 Thus, for every $\gamma \in \widehat{K}$, it follows that 
 \begin{align*}
&\int_{\mathbb{R}^n}\int_{\mathbb{R}^n}|f_{\gamma}(x)|\ |\widehat{f_{\gamma}}(\xi)|e^{2\pi|x\cdot\gamma|}\ dx\ d\xi &\\
& \leq \int_{\mathbb{R}^n}\int_{\mathbb{R}^n}\int_{K}|f(x,k)|\  \| \xi \otimes \gamma (f)\|_{\text{HS}}\ e^{2\pi|x\cdot\gamma|}\ dx\ dk\ d\xi  < \infty.
\end{align*}
Hence, using Beurling theorem for $\mathbb{R}^n$, we get $f_{\gamma} = 0$ a.e. Since $\gamma \in \widehat{K}$ is arbitrary, therefore using Lemma \ref{zeroall}, we can conclude  that $f = 0$ a.e. 
\end{proof}
\begin{lem}\label{ber-open}
 Let $M = \mathbb{R}^n \times H$ be an open subgroup of $G$. If $f\in L^1({G})$ satisfies  conditions of Theorem \ref{ber}, then so does $f|_{M}$.
\end{lem}
 \begin{proof}
Since $\widehat{S/H}$ is compact and $\widehat{\widehat{S/H}}$ is identified with $S/H$ \cite[Theorem 24.2]{Hew:Ros:63:70}, therefore  we have
\[\int_{\widehat{S/H}}\overline{\eta(x)}\ d\eta = \left\{\begin{array}{cl}
 0  \text{ if } x \notin H \\ 1, \text{ if } x\in H.
\end{array}\right. \]
Thus,
\begin{align*}
\int_{\widehat{S/H}}\widehat{f}(\xi, \chi\eta)\ d\eta 
&= \int_{\mathbb{R}^n}\int_{S}f(x,s)e^{-2\pi i \xi x}\ \overline{\chi(s)}\ \Big( \int_{\widehat{S/H}}\overline{\eta(s)}d\eta\Big)\ dx\ ds\\
&\int_{\mathbb{R}^n}\int_{H}f(x,s)e^{-2\pi i \xi x}\ \overline{\chi(s)}\ dx\ ds = \widehat{f|_{M}}(\xi, \chi|_M).
\end{align*} 
 Therefore,
\begin{align*}
&\int_{\mathbb{R}^n} \int_{H} \int_{\mathbb{R}^n} \int_{\widehat{H}}|f|_{M}(x,h)|\ |\widehat{f}|_{M}(\xi, \chi)|\ e^{2\pi|x\cdot \xi|}\ dx\ dh \ d\xi\ d\chi\\
&= \int_{\mathbb{R}^n} \int_{H} \int_{\mathbb{R}^n} \int_{\widehat{H}}|f|_{M}(x,h)|\ |\int_{\widehat{S/H}}\widehat{f}(\xi, \chi\eta)\ d\eta|\ e^{2\pi|x\cdot \xi|}\ dx\ dh \ d\xi\ d\chi\\
&\leq \int_{\mathbb{R}^n} \int_{H} \int_{\mathbb{R}^n} \int_{\widehat{H}} \int_{\widehat{S/H}} |f|_{M}(x,h)|\ |\widehat{f}(\xi, \chi\eta)| \ e^{2\pi|x\cdot \xi|}\ dx\ dh \ d\xi\ d\chi \ d\eta\\ 
& \leq \int_{\mathbb{R}^n} \int_{S} \int_{\mathbb{R}^n} \int_{\widehat{S}} |f(x,h)|\ |\widehat{f}(\xi, \chi\eta)| \ e^{2\pi|x\cdot \xi|}\ dx\ dh \ d\xi\ d\chi \ < \infty.\qedhere
\end{align*}
\end{proof}  
\noindent Using Lemma \ref{comp2} and Lemma \ref{ber-open}, we have the proof of Theorem \ref{ber}. 

\begin{proof}
Let $s\in S$  be arbitrarily. If $f\in L^1\cap L^2(G) $ satisfies the condition of Theorem \ref{ber}, then so does $f_{s}$, where $f_{s}(x,t)=f(x,st).$ Since $S$ has compact open subgroup $K$, therefore using Lemma \ref{comp2} and Lemma \ref{ber-open}, we get $f_{s}|_{\mathbb{R}^n\times K} = 0$ a.e. Thus, we get  $f=0$ a.e.
\end{proof}

In the next result, we give a Beurling theorem  version for Gabor transform on abelian groups by reducing it to Fourier transform case. 
\begin{thm}\label{abel}
Let $f\in L^2(G)$ and $\psi$ be a window function such that 
\begin{align*}
\int_{\mathbb{R}^n}\int_{S}\int_{\mathbb{R}^n}\int_{\widehat{S}}|G_{\psi}f(x,s,\xi,\sigma)|\ e^{\pi(\|x\|^2+\|\xi\|^2)/2}\ dx\ ds\ d\xi\ d\sigma < \infty.
\end{align*}
Then either $f=0$ a.e. or $\psi=0$ a.e.
\end{thm}
\begin{proof}
For $(x,k), (z,t) \in \mathbb{R}^n \times S$ and $(\xi,\gamma),(\zeta,\chi) \in \widehat{\mathbb{R}^n}\times \widehat{S}$,  define
\begin{flalign*}
&&F_{(z,t,\zeta,\chi)}(x,k,\xi,\gamma)&=e^{2\pi i \xi x}\ \gamma(k)\ G_\psi (M_{\zeta,\chi}T_{z,t}f)(x,k,\xi,\gamma)\ &\\
&&&\qquad \times G_\psi (M_{\zeta,\chi}T_{z,t}f)(-x,k^{-1},-\xi,\gamma^{-1}). &
\end{flalign*}
The function $F_{(z,t,\zeta,\chi)}$ is continuous and is in $L^1\cap L^2(\mathbb{R}^n \times S\times \widehat{\mathbb{R}^n}\times \widehat{S})$. Moreover, on using \cite[Lemma 3.2]{ash:18}, we have
\begin{flalign}\label{four:val}
\widehat{F_{(z,t,\zeta,\chi)}}(\omega,\delta,y,v)=F_{(z,t,\zeta,\chi)}(-y,v^{-1},\omega,\delta).
\end{flalign}
Using (\ref{eq1}), $F_{(z,t,\zeta,\chi)}(x,k,\xi,\gamma)$ can be written as 
\begin{flalign}
&F_{(z,t,\zeta,\chi)}(x,k,\xi,\gamma)\nonumber &\\
&=e^{2\pi i \xi x}\ \gamma(k)\ e^{-2\pi i (\xi-\zeta)z}\ (\chi^{-1}\gamma)(t^{-1})\ G_\psi f(x-z,t^{-1}k,\xi-\zeta,\chi^{-1}\gamma) \nonumber &\\
&\label{value}\quad \times e^{-2\pi i (-\xi-\zeta)z}\ (\chi^{-1}\gamma^{-1})(t^{-1})\ G_\psi f(-x-z,t^{-1}k^{-1},-\xi-\zeta,\chi^{-1}\gamma^{-1}).&
\end{flalign}
Applying (\ref{four:val}) and (\ref{value}), we have \begin{align*}
&\int_{\mathbb{R}^n}\int_{S}\int_{\mathbb{R}^n}\int_{\widehat{S}} \int_{\mathbb{R}^n}\int_{\widehat{S}}\int_{\mathbb{R}^n}\int_{S}|F_{(z,t,\zeta,\chi)}(x,k,\xi,\gamma)|\ |\widehat{F_{(z,t,\zeta,\chi)}}(\omega,\delta,y,v)|\\
&\qquad \qquad \qquad \hspace{1.5cm} \times  e^{2\pi|x\cdot  \omega+ \xi\cdot y|}\ dx\ dk\ d\xi\ d\gamma\ d\omega\ d\delta\ dy\ dv\\
&\int_{\mathbb{R}^n}\int_{S}\int_{\mathbb{R}^n}\int_{\widehat{S}}\int_{\mathbb{R}^n}\int_{\widehat{S}}\int_{\mathbb{R}^n}\int_{S}|F_{(z,t,\zeta,\chi)}(x,k,\xi,\gamma)|  |F_{(z,t,\zeta,\chi)}(-y,v^{-1},\omega,\delta)| \ \\
&\qquad \qquad \qquad \hspace{1.5cm} \times e^{\pi(\|x\|^2+\|\xi\|^2+{\|\omega\|}^2+\|y\|^2)}\ dx\ dk\ d\xi\ d\gamma\ d\omega\ d\delta\ dy\ dv \\
&=\left(\int_{\mathbb{R}^n}\int_{S}\int_{\mathbb{R}^n}\int_{\widehat{S}}|F_{(z,t,\zeta,\chi)}(x,k,\xi,\gamma)|e^{\pi(\|x\|^2+\|\xi\|^2)}\ dx\ dk\ d\xi\ d\gamma\right)^2\\
  &= \left( \int_{\mathbb{R}^n}\int_{S}\int_{\mathbb{R}^n}\int_{\widehat{S}}  |G_{\psi}f(-x-z,t^{-1}k^{-1}, -\xi-\zeta, \gamma^{-1}\chi^{-1})| \right.\\
&\quad\quad \hspace{2cm} \quad  \times \left.  |G_{\psi}f(x-z,t^{-1}k, \xi-\zeta, \gamma\chi^{-1})| e^{(\|x\|^2+\|\xi\|^2)}\ dx\ dk\ d\xi\ d\gamma \right)^2\\
&= \left(\int_{\mathbb{R}^n}\int_{S}\int_{\mathbb{R}^n}\int_{\widehat{S}}|G_{\psi}f(-x-2z,t^{-2}k^{-1}, -\xi-2\zeta, \gamma^{-1}\chi^{-2})|\ |G_{\psi}f(x,k, \xi, \gamma)| \right.\\
&\hspace{3cm}\times\left. e^{\pi(\|x+z\|^2+\|\xi+\zeta\|^2)}\ dx\ dk\ d\xi\ d\gamma \right)^2\\ 
&= e^{2\pi(\|z\|^2+\|\zeta\|^2)}(H\ast H(-2z,t^{-2},-2\xi,{\gamma}^{-2}))^2 < \infty,
 \end{align*}
 where $H(x,s,\xi,\sigma)=|G_{\psi}f(x,s,\xi,\sigma)|e^{\pi(\|x\|^2+\|\xi\|^2)/2}$.
Thus, using Theorem \ref{ber}, it follows that  $F_{(z,t,\zeta,\chi)}\equiv 0$ for all $(z,t,\zeta,\chi)$. Since, $$F_{(-z,t^{-1},-\zeta,\chi^{-1})}(0,e,0,I)=e^{4\pi i \zeta z}\ \chi(t)^2\ (G_\psi f(z,t,\zeta,\chi))^2,$$ 
therefore, $G_\psi f\equiv 0$ which using (\ref{GT-norm})  implies that  either $f=0$ a.e. or $\psi = 0$ a.e.
\end{proof}
\noindent  We shall next prove the Beurling's theorem for Gabor transform for the groups of the form $\R^n \times K$, when $K$ is a compact group. 
\begin{thm} \label{comp}
Let $f, \psi  \in L^2(\R^n \times K)$, where $K$ is a compact group  such that 
\begin{align*}
\int_{\mathbb{R}^n}\int_{\mathbb{K}}\int_{\mathbb{R}^n} \sum_{\gamma \in \widehat{K}}\|G_{\psi}f(x,k,\xi,\gamma)\|_{\text{HS}}\ e^{\pi(\|x\|^2+\|\xi\|^2)/2}\ dx\ dk\ d\xi\ d\gamma\ < \infty.
\end{align*}
Then either $f=0$ a.e. or $\psi =0$ a.e. 

\end{thm}
\begin{proof}
Assume that $\psi \neq 0$. For $\omega,\gamma \in \widehat{K}$, let $\ch_\omega$ and $\ch_\gamma$ be the Hilbert spaces of dimensions $d_\omega$ and $d_\gamma$ with orthonormal bases $\{e_{i}^{\omega}\}_{i=1}^{d_{\omega}}$ and $\{e_{i}^{\gamma}\}_{i=1}^{d_{\gamma}}$ respectively. \\
For fixed $e_{r}^{\gamma}, e_{s}^{\gamma}$, we define $\tau : \R^n \rightarrow \C$ by
\begin{flalign*}
\tau(x)=\int_{K}{\psi(x,k)\ \overline{\langle{\gamma(k)^\ast e_{r}^{\gamma},e_{s}^{\gamma}}\rangle}}\ dk.
\end{flalign*}
Using the H\"{o}lder's inequality, it follows that $\tau \in L^2(\R^n)$. By Lemma \ref{zeroall}, we fix $\gamma \in \widehat{K}$ for which $\tau \neq 0$. For $\sigma \in \widehat{K}$, we can write
 \begin{flalign}
&&\gamma(k)e_{r}^{\gamma}& =\sum_{j=1}^{d_\gamma}C_{j,r}^ke_{j}^{\gamma} \nonumber&\\
&\text{and}& \gamma \otimes \sigma &=\sum_{\delta \in K_\sigma}{m_\delta\ \delta}, \label{tensor}
\end{flalign}
$K_\sigma$ is a finite subset of $\widehat{K}$ and $C_{j,r}^k$'s, $m_\delta$'s are scalars (see \cite{Hew:Ros:63:70}). For fixed $e_{p}^{\omega}$ and $ e_{q}^{\omega}$, we define $g : \R^n \rightarrow \C$ such that
\begin{flalign*}
g(x)=\int_{K}{f(x,k)\ \ \overline{\langle{\omega(k)^\ast e_{p}^{\omega},e_{q}^{\omega}}\rangle}}\ dk.
\end{flalign*}
Clearly, $g \in L^2(\R^n)$. Consider a function $\varphi : \R^n \times K \rightarrow \C$ defined by
\begin{flalign*}
\varphi (x,k) = \psi(x,k)\ \overline{\langle{\gamma(k)^\ast e_{r}^{\gamma},e_{s}^{\gamma}}\rangle}.
\end{flalign*}
Then, $\varphi \in L^2(\R^n \times K)$ and $G_{\varphi}f(x,k,\xi,\sigma)$ is a Hilbert-Schmidt operator for all $(x,k)\in \R^n \times K$ and for almost all $(\xi,\sigma)\in \widehat{\R^n} \times \widehat{K}$. \\
For $\sigma \in \widehat{K}$ and fixed $e_{l}^{\sigma}, e_{m}^{\sigma}$, using  \cite{ash:18} we have 
\begin{flalign*}
&\langle{G_{\varphi}f(x,k,\xi,\sigma)e_{l}^{\sigma},e_{m}^{\sigma}}\rangle =\sum_{j=1}^{d_\gamma}\sum\limits_{\delta \in K_\sigma}{C_{j,r}^k\ m_\delta\ \langle{G_{\psi}f(x,k,\xi,\delta)e_{l,j}^{\delta},e_{m,s}^{\delta}}\rangle}.
\end{flalign*}
Let  $M_\sigma=\max{\{ |m_\delta| : \delta \in K_\sigma\}}$. As $|K_{\sigma}|\leq d_{\gamma}d_{\sigma} < \infty$, we have $M_{\sigma}<\infty.$ Using  Cauchy-Schwarz inequality, we have
\begin{flalign*}
\|G_{\varphi}f(x,k,\xi,\sigma)\|_{\text{HS}}^2 &=\sum_{l,m=1}^{d_\sigma}{|\langle{G_{\varphi}f(x,k,\xi,\sigma)e_{l}^{\sigma},e_{m}^{\sigma}}\rangle|^2} &\\
&\leq \sum_{l,m=1}^{d_\sigma}{\Big(\sum_{j=1}^{d_\gamma}\sum\limits_{\delta \in K_\sigma}{|C_{j,r}^k\ m_\delta\ \langle{G_{\psi}f(x,k,\xi,\delta)e_{l,j}^{\delta},e_{m,s}^{\delta}}\rangle|}\Big)^2} &\\
&\leq \sum_{l,m=1}^{d_\sigma}{M_{\sigma}^2\ |K_\sigma|\ d_\gamma\ \Big(\sum_{j=1}^{d_\gamma}\sum\limits_{\delta \in K_\sigma}{|\langle{G_{\psi}f(x,k,\xi,\delta)e_{l,j}^{\delta},e_{m,s}^{\delta}}\rangle|^2}\Big)} &\\
&\leq \sum_{l,m=1}^{d_\sigma}{M_{\sigma}^2\ |K_\sigma|\ d_\gamma\  \sum_{j=1}^{d_\gamma} \sum\limits_{\delta \in K_\sigma}\|G_{\psi}f(x,k,\xi,\delta)\|_{\text{HS}}^2} &\\*
&\leq  d_\sigma^2\ M_{\sigma}^2\  |K_\sigma|\ d_\gamma^2\ \Big( \sum\limits_{\delta \in K_\sigma}\  \|G_{\psi}f(x,k,\xi,\delta)\|_{\text{HS}}\Big)^2. &
\end{flalign*}
So, it follows that
\begin{flalign}
\|G_{\varphi}f(x,k,\xi,\sigma)\|_{\text{HS}} &\leq C_{\sigma,\gamma} \sum\limits_{\delta \in K_\sigma} \|G_{\psi}f(x,k,\xi,\delta)\|_{\text{HS}},  \label{step1:bdd}
\end{flalign}
where $C_{\sigma,\gamma} =   d_\sigma\ M_{\sigma}\ |K_\sigma|\ d_\gamma\  $ a constant depending on $\sigma$ and $\gamma$.
Now for  every $\sigma \in  \widehat{K}$, using (\ref{step1:bdd}), we obtain 
\begin{align}
&\int_{\mathbb{R}^n}\int_{K}\int_{\mathbb{R}^n}\|G_{\varphi}f(x,k,\xi,\sigma)\|_{\text{HS}}\  e^{\pi(\|x\|^2+\|\xi\|^2)/2}\ dx\ dk\ d\xi \nonumber \\
&\leq C_{\sigma,\gamma}\int_{\mathbb{R}^n}\int_{K}\int_{\mathbb{R}^n} \sum\limits_{\delta \in K_\sigma} \|G_{\psi}f(x,k,\xi,\delta)\|_{\text{HS}}\ e^{\pi(\|x\|^2+\|\xi\|^2)/2}\ dx\ dk\ d\xi<\infty. \label{in-norm}
\end{align}
For $x, \xi \in \mathbb{R}^n $, the function $G_{\tau}g$ is given by
\begin{flalign*}           
G_{\tau}g(x,\xi)
&=\int_{K}\langle G_{\varphi}f(x,k,\xi,\omega)e_{p}^{\omega}, e_{q}^{\omega}\rangle\ dk.
\end{flalign*}
Thus,
\begin{flalign*}
|G_{\tau}g(x,\xi)|&\leq \int_{K} \|G_{\varphi}f(x,k,\xi,\omega)\|_{\text{HS}}\ dk.\end{flalign*}
On using (\ref{in-norm}), it follows 
 \begin{flalign*}
&\int_{\mathbb{R}^n} \int_{\mathbb{R}^n}|G_{\tau}g(x,\xi)|e^{\pi(\|x\|^2+\|\xi\|^2)/2}\ dx\ d\xi\\
&\leq \int_{\mathbb{R}^n} \int_{\mathbb{R}^n} \int_{K} \|G_{\varphi}f(x,k,\xi,\omega)\|_{\text{HS}}\ e^{\pi(\|x\|^2+\|\xi\|^2)/2}\ dx \ d\xi\ dk < \infty.
\end{flalign*} 
Then by Beurling theorem for Gabor transform on $\R^n$ (see \cite{incollection}) or Theorem  \ref{abel}, we conclude that $g=0$ a.e. 
Since $\omega \in \widehat{K}$ is arbitrary, therefore using Lemma \ref{zeroall}, we get $f=0$ a.e.
\end{proof}
\section*{Acknowledgement}
\noindent The first author is supported by UGC under joint UGC-CSIR  Junior Research Fellowship (Ref. No:21/12/2014(ii)EU-V).\\

\bibliographystyle{amsplain}

\end{document}